\documentclass[reqno,11pt]{amsart}
\usepackage{amssymb,amsmath}
\usepackage{amscd}

\usepackage{color}

\theoremstyle{plain}
\newtheorem{theorem}{Theorem}[section]
\newtheorem{proposition}[theorem]{Proposition}

\newtheorem{lemma}[theorem]{Lemma}
\theoremstyle{definition}
\newtheorem{definition}[theorem]{Definition}

\theoremstyle{remark}
\newtheorem{remark}[theorem]{Remark}
\numberwithin{equation}{section}

\newcommand{\la}{\lambda}
\newcommand{\al}{\alpha}

\newcommand{\R}{\mathbb R}
\newcommand{\C}{\mathbb C}

\newcommand{\eps}{\varepsilon}

\begin{document}

\bibliographystyle{plain}

\title[Helicoidal solutions of a fractional Allen-Cahn equation]{Solution of the fractional Allen-Cahn equation which are invariant under screw motion}


\author{Eleonora Cinti, Juan Davila, and Manuel Del Pino}

\address{E.C., Wierstrass Institute for Applied Analysis and Stochastics,
Mohrenstr. 39,
10117 Berlin (Germany) }

\email{cinti@wias-berlin.de}

\address{J. D. and M. DP, Departamento de Ingenieria Matem\'atica
Facultad de Ciencias Fisicas y Matem\'aticas
Universidad de Chile. Casilla 170, Correo 3, Santiago, Chile.   }

\email{jdavila@dim.uchile.cl}
\email{delpino@dim.uchile.cl }
\maketitle

\begin{abstract}
We establish existence and non-existence results for entire solutions to the fractional Allen-Cahn equation in $\R^3$, which vanish on helicoids and are invariant under screw-motion.  In addition, we prove that helicoids are surfaces with vanishing nonlocal mean curvature.
\end{abstract}

\section{Introduction}

In this paper we establish existence results for a class of entire solutions to the fractional Allen-Cahn equation
\begin{equation}\label{AC}
(-\Delta)^\al u=F'(u)\quad \mbox{in}\;\R^n,
\end{equation}
where $F$ is a double-well potential, i.e. it satisfies the following properties:
\begin{itemize}
\item $t\mapsto F(t)$ is an even, positive function of class $C^{2,\gamma}$, with $\gamma>\max\{0,1-2\alpha\}$,
\item $F(t)\geq F(\pm 1)$ and equality holds if and only if $t=\pm 1$.
\end{itemize}

Moreover, we assume also that
\begin{equation}\label{hp-f}
F''(0)<0\quad \mbox{and}\quad F''(0)t\leq F'(t)\;\; \mbox{for every}\:\: t\geq 0.
\end{equation}
A classical example of such potential is $F(t)=\frac{1}{4}(1-t^2)^2$.

In the last years, there has been much interest in the study of existence of solutions to the fractional Allen-Cahn equation. In \cite{CS1,CS2,C-SM}, the existence of layer-type solutions, that is solutions monotone in one direction with limits $\pm 1$ at $\pm \infty$, has been established, while in \cite{Cinti}, one of the authors proved existence for saddle-shaped solution, that are solutions which vanish on the Simons cone
$$\mathcal C=\{(x,\xi)\in \R^{m}\times \R^m\::\:|x|=|\xi|\},$$
they are odd with respect to $\mathcal C$ and even with respect to the coordinate axis.
In all these works, the proof of existence relies on a variational argument, which makes use of the symmetries of the problem.
Also in this paper we are interested in solutions of \eqref{AC} which satisfies some symmetry properties, and the technique we use relies on the variational structure of \eqref{AC}.

More precisely we establish existence and non-existence results for solutions which vanish on helicoids and are invariant under screw motion (see \eqref{heli} and \eqref{screw} below for precise definitions).
For the classical Allen-Cahn equation, analogue results are contained in a work by Musso, Pacard and one of the authors \cite{DPMP}.

The interest in the study of solutions which vanish on helicoids comes from the fact that helicoids are surfaces with zero mean curvature.

There is a very strict connection between the Allen-Cahn equation and the classical theory of minimal surfaces. The classical result by Modica and Mortola \cite{MM} establishes that the energy functional associated to the classical Allen-Cahn equation, after a suitable rescaling, $\Gamma$-converges to the Perimeter functional.

In the fractional setting, an analougue $\Gamma$-convergence type result has been established in \cite{ABS,MdM} for powers $1/2\leq \alpha<1$, and in \cite{SV} for any power $0<\alpha<1$: after a suitable rescaling, the energy functional associated to the fractional Allen-Cahn equation $\Gamma$-converges to the classical perimeter if $1/2\leq \alpha<1$ and to the \emph{nonlocal perimeter} if $0<\alpha<1/2$. The notion of nonlocal perimeter has been introduced by Caffarelli, Roquejoffre and Savin in \cite{CRS}, where existence, regularity results and a monotonicity formula for nonlocal minimal surfaces have been established. Similarly to the case of classical perimeter, performing the first variation of the nonlocal perimeter functional, one can define the notion of \textit{nonlocal mean curvature} (see \eqref{NMC} below).

In this paper, we focus our attention both on the PDE problem and on the geometric one. Indeed in Theorems \ref{exist} and \ref{nonexist} below we establish existence and non-existence results for solutions to \eqref{AC} (in the case of space dimension $n=3$) which vanish on helicoids and are invariant under screw motion. Moreover, in Theorem \ref{helicoid} we prove that helicoids have zero nonlocal mean curvature.

We recall now the definition of helicoid and screw motion.
We will work in dimension $n=3$. Given $\lambda >0$, the \emph{helicoid} $H_\la$ is the minimal surface which can be parametrized in the following way
\begin{equation}\label{heli}
\R\times \R \ni (t,\theta)\mapsto \left(te^{i\theta},\frac{\lambda}{\pi}\theta\right)\in \C\times\R=\R^3.
\end{equation}
The screw motion of parameter $\la$ acting on $\C\times \R$ is given by
\begin{equation}\label{screw}
\sigma_\la^\beta(z,s)=\left(e^{i\beta}z,s+\frac{\la}{\pi}\beta\right).
\end{equation}
Obviously $H_\la$ is invariant under the action of $\sigma_\la^\beta$ for every $\beta \in \R$.

Our first main result is the construction of a nontrivial entire solution to \eqref{AC} in dimension $3$ which vanishes on $H_\la$, provided $\la$ is chosen sufficiently large.
We define
\begin{equation}\label{lambda}
\lambda_*:=\frac{\pi}{(-F''(0))^{\frac{1}{2\al}}}.
\end{equation}
\begin{theorem}\label{exist}
Let $n=3$.
Assume that $F$ is a double-well potential satisfying \eqref{hp-f} and that $\la>\la_* $. Then, there exists a solution of the fractional Allen-Cahn equation \eqref{AC} whose zero set is equal to $H_\la$. This solution is invariant under the screw motion of parameter $\lambda$, i.e.
$$u \circ \sigma_\la^\beta = u,$$
for every $\beta \in \R$.
\end{theorem}
In the following result, we also prove that Theorem \ref{exist} is, in some sense, sharp.
\begin{theorem}\label{nonexist}
Suppose that $\la\leq \la_*$. Then, there are no nontrivial bounded solutions of \eqref{AC}, which vanish on the helicoid $H_\la$ and are invariant under the screw motion of parameter $\lambda$.
\end{theorem}
To prove the previous results, we will realize the nonlocal problem (\ref{AC}) as a
local problem in $\R^{n+1}_+$ with a nonlinear Neumann condition
on $\partial \R_{+}^{n+1}=\R^n$ (the so called Caffarelli-Silvestre extension \cite{CS}). More precisely, if $u=u(x)$ is
a function defined on $\R^n$, we consider its $s$-\textit{harmonic extension}
$v=v(x,y)$ in $\R^{n+1}_+=\R^n\times(0,+\infty)$. It is
well known (see \cite{C-SM,CS}) that $u$ is a solution of
(\ref{AC}) if and only if $v$ satisfies
\begin{equation}\label{AC2}
\begin{cases}
\mbox{div}(y^{1-2\al} \nabla v) =0& \text{in}\; \R_{+}^{n+1},\\
- \frac{1}{c_\alpha}\lim_{y\rightarrow 0}y^{1-2\al}\partial_{y}v=F'(v)& \text{on}\; \R^{n}=\partial
\R_{+}^{n+1},
\end{cases}
\end{equation}
where $c_\alpha$ is the constant given in \eqref{c}.

The energy associated to problem \eqref{AC2} is
\begin{equation}\label{en}
E(v)=\frac{1}{2c_\alpha}\int_0^{+\infty}\int y^{1-2\al} |\nabla v|^2 dx dy +\int F'(v(x,0)) dx.
\end{equation}
Both proofs of Theorems \ref{exist} and \ref{nonexist} follows the ideas contained in \cite{DPMP} for the local case. However, due to the nonlocality of our problem, we have to deal with some difficulties. In particular, in order to prove the existence result, we will need to establish an energy estimate for minimizers, which will ensures us that the limit of a minimizing sequence does not identically vanish. In this step we follows the technique used in \cite{Cinti},  based on a suitable choice of a cut-off function. On the other hand, for the non-existence result, a crucial ingredient will be an exponential decay in the $y$-variable of the extended solution $v$ of \eqref{AC2}, under our symmetry assumption (see Proposition \ref{exp-decay} below).
 
We recall now the notions of nonlocal minimal surface and of nonlocal mean curvature.
Nonlocal minimal surfaces were introduced in \cite{CRS} as boundaries of measurable sets $E$ whose characteristic function $\chi_E$ minimizes an $H^\alpha$-norm. More precisely, for any $0<\alpha<1/2$, the nonlocal $\alpha$-perimeter functional of a set $E$ in an open set $\Omega\subset \R^n$, is given by
\begin{equation}\label{def-per}\text{Per}_{2\alpha}(E,\Omega):=L(E\cap \Omega,\R^n\setminus E) + L(E\setminus \Omega, \Omega \setminus E),\end{equation}
where, for two disjoint measurable sets $A$ and $B$, $L(A,B)$ denotes the quantity
$$L(A,B):=\int_A\int_B\frac{1}{|x-\bar x|^{n+2\alpha}}dx d\bar x.$$
A set $E$ is said to be $\alpha$-minimal in $\Omega$ if
$$\text{Per}_{2\alpha}(E,\Omega)\leq \text{Per}_{2\alpha}(F,\Omega)$$
for any measurable set $F$ with $E\triangle F \subset\subset \Omega$. Notice that in the literature (see e.g. \cite{CRS}) the fractional $s$-perimeter is defined for any $s\in (0,1)$ and corresponds to \eqref{def-per} for $s=2\alpha$. Here, we prefer to keep this notation to be consistent with the fractional power of the Laplacian.

Analougsly to the classical theory of minimal surfaces, performing the first variation of the nonlocal perimeter functional, we end up with the notion of \emph{nonlocal mean curvature}. More precisely, the Euler-Lagrange equation for $\text{Per}_{2\alpha}(E,\R^n)$ is given by
\begin{equation}\label{NMC}
\mathcal H^{2\alpha}_E(x):=\int_{\R^n} \frac{\chi_E(\bar x)-\chi_{E^C}(\bar x)}{|x-\bar x|^{n+2\alpha}} d\bar x=0,
\end{equation}
where $E^C:=\R^n\setminus E$ and $\mathcal H^{2\alpha}_E$ denotes the nonlocal mean curvature of the set $E$ (we write NMC for short).
Recently, there has been much interest in the study of surfaces with vanishing or constant nonlocal mean curvature.
In \cite{DDW}, two of the authors and J. Wei provide examples of surfaces with zero NMC. More precisely, they establish existence of nonlocal minimal Lawson cones (for any $0<\alpha<1/2$) and proved their stability in dimension $7$ for $\alpha$ small.  Moreover, for $\alpha \rightarrow 1/2$, they constructed the nonlocal analogue of catenoids.
Concerning the study of surfaces with constant NMC, in \cite{CFSMW, CFMN} the analogue of Alexandrov Theorem, charachterizing spheres as the only closed embedded hypersurfaces in $\R^n$ with constant mean curvature, has been established. In \cite{CFSMW} the existence of Delaunay-type surfaces (in the 2-dimensional case) is established, while in \cite{DDDV} periodic and cylindrical symmetric hypersurfaces, which
minimize a certain fractional perimeter under a volume constraint, are considered.

In our last result we provide a new example of surface with zero NMC: we prove that helicoids, which have zero (classical) mean curvature, also have zero NMC.

\begin{theorem}\label{helicoid}
For any $\lambda>0$ and for any $0<\alpha <1/2$, we have that
$$\mathcal H^{2\alpha}_{H_\la} \equiv 0.$$
\end{theorem}
The paper is organized as follows:
\begin{itemize}
\item in Section 2, we recall some preliminaries on the fractional Laplacian in a bounded domain with $0$-Dirichlet boundary condition.
\item in Section 3, we establish existence and non existence results for \eqref{AC} in the 1-dimensional case.
\item in Section 4, we prove our existence result Theorem \ref{exist}.
\item in Section 5, we give the proof of Theorem \ref{nonexist}.
\item in Section 6, we prove Theorem \ref{helicoid}.
\end{itemize}

\section{Preliminaries}

In this section we recall some well known facts about the fractional Laplacian in a bounded domain (see \cite{BCdP,CT}).

Let $\Omega$ be a sufficiently regular (say Lipschitz) domain in $\R^n$. We denote by $(-\Delta)^\alpha$ the fractional power of the Laplacian $-\Delta$ in $\Omega$ with zero Dirichelt boundary condition on $\partial \Omega$.

To define $(-\Delta)^\alpha$, let us consider $\{\mu_k,\zeta_k\}_{k=1}^\infty$ the eigenvalues and corresponding eigenfunctions of the Laplacian $-\Delta$ in $\Omega$ with zero Dirichlet boundary condition:
$$\begin{cases}
-\Delta \zeta_k=\mu_k\zeta_k &\mbox{in}\;\;\Omega\\
\zeta_k=0 &\mbox{on}\;\;\partial \Omega.
\end{cases}
$$
Let $u=\sum_{k=1}^\infty a_k\zeta_k$, then we define
\begin{equation}\label{spect}
(-\Delta)^\alpha u=\sum_{k=1}^\infty a_k\mu_k^{\alpha}\zeta_k.
\end{equation}
Let now $\mathcal C_\Omega$ be the cylinder $\mathcal C_\Omega =\Omega \times (0,\infty)$ and $\partial_L \mathcal C_\Omega=\partial \Omega\times (0,\infty)$ its lateral boundary.
Following \cite{BCdP,CT}, we can consider the extension operator with zero Dirichlet boundary condition on the all $\partial_L\mathcal C_\Omega$.
\begin{definition}
We define the $\alpha$-harmonic extension $v=\mathcal E_\alpha(u)$ in $\mathcal C_\Omega$ of a function $u$ defined in $\Omega$ and vanishing on $\partial \Omega$ as the solution of the problem
\begin{equation}\label{Dir-ext}
\begin{cases}
 \mbox{div}(y^{1-2\alpha}\nabla v)=0 & \mbox{in}\;\;\mathcal C_\Omega\\
v=0 & \mbox{on}\;\; \partial_L \mathcal C\\
v=u &\mbox{on}\;\; \Omega \times \{y=0\}.
\end{cases}
\end{equation}
\end{definition}
It is well known that (see \cite{BCdP,CT,CS})
$$(-\Delta)^\alpha u(x)=-\frac{1}{c_\alpha}\lim_{y\rightarrow 0}y^{1-2\alpha}\partial_yv(x,y),$$
where
\begin{equation}\label{c}
c_\alpha=\frac{2^{1-2\alpha}\Gamma(1-\alpha)}{\Gamma(\alpha)}.
\end{equation}

We recall now the explicit expression (see Lemma 3.4 in \cite{BCdP}) for $\mathcal E_\alpha(u)$ in terms of the spectral decomposition \eqref{spect}.
\begin{lemma}[Lemma 3.4 in \cite{BCdP}]\label{spect-ext}
Let $\{\mu_k,\zeta_k\}$ be, as before, the eigenvalues and eigenfunctions of $-\Delta$ in $\Omega$ (with zero Dirichlet boundary condition). Let $u=\sum_{k=1}^\infty a_k \zeta_k$ be such that $\sum_{k=1}^\infty a_k \mu_k^\alpha < \infty$. Then, the $\alpha$-harmonic extension of $u$ is given by
$$v(x,y)=\mathcal E_\alpha(u)(x,y)=\sum_{k=1}^\infty a_k \zeta_k(x) \varphi(\mu_k^{\frac{1}{2}}y),$$
where $\varphi$ is a solution of the problem
\begin{equation}\label{ODE}
\begin{cases}
\varphi''+\frac{1-2\alpha}{y}\varphi' -\varphi=0 & \mbox{for}\;\;y>0\\
-\lim_{y\rightarrow 0}y^{1-2\alpha}\varphi '(y)=c_\alpha\\
\varphi(0)=1.
\end{cases}
\end{equation}
\end{lemma}
The solution $\varphi$ coincides with the solution of the following problem
$$\varphi''+\frac{1-2\alpha}{y}\varphi' -\varphi=0,\quad \varphi(0)=1,\quad \lim_{y\rightarrow \infty}\varphi(y)=0,$$
and minimizes the functional
$$\int_0^\infty y^{1-2\al}\big( |\varphi(y)|^2+ |\varphi'(y)|^2\big) dy.$$
Moreover, it is a combination of Bessel functions, as shown in the following lemma.
%
\begin{lemma}[Lemma 2.2 in \cite{BGS}]\label{autofun1}
The solution of the ODE
\begin{equation}\label{eqy}
\varphi''+\frac{1-2\al}{y}\varphi' -\varphi=0\\
\end{equation}
may be written as $\varphi(y)=y^\al\psi(y)$, where $\psi$ solves the well known Bessel equation
\begin{equation}\label{bessel}
y^2\psi''+y\psi'-(y^2+\al^2)\psi=0.
\end{equation}
In addition \eqref{bessel} has two linearly independent solutions, $I_\al$, $Z_\al$, which are the modified Bessel functions; their asymptotic behaviour is given precisely by
\begin{align}\label{asym}
&I_\al(y)\sim\frac{1}{\Gamma(\al+1)}\left(\frac{y}{2}\right)^\al\left( 1+\frac{y^2}{4(\al+1)}+\frac{y^4}{32(\al+1)(\al+2)}+...\right).\nonumber\\
&Z_\al(y)\sim\frac{\Gamma(\al)}{2}\left(\frac{2}{y}\right)^\alpha\left( 1+\frac{y^2}{4(1-\al)}+\frac{y^4}{32(1-\al)(2-\al)}+...\right)+\nonumber\\
&\hspace{4em} + \frac{\Gamma(-\al)}{2^\al}\left(\frac{y}{2}\right)^\al\left( 1+\frac{y^2}{4(\al+1)}+\frac{y^4}{32(\al+1)(\al+2)}+...\right),
\end{align}
for $y\rightarrow 0^+$, $\al \notin \mathbb Z$. And when $y \rightarrow +\infty$,
\begin{align}\label{asym-infty}
I_\al(y)\sim \frac{1}{\sqrt{2\pi y}}e^y \left( 1-\frac{4\al^2-1}{8y}+\frac{(4y^2-1)(4y^2-9)}{2!(8y)^2}+...\right),\nonumber \\
Z_\al(y)\sim\sqrt{\frac{\pi}{2y}}e^{-y} \left( 1-\frac{4\al^2-1}{8y}+\frac{(4y^2-1)(4y^2-9)}{2!(8y)^2}+...\right).
\end{align}
\end{lemma}
In the sequel, we will use both type of solutions given in Lemma \ref{autofun1} above (one is growing exponentially as $y\rightarrow \infty$, one is decaying exponentially to $0$). Up to a normalization constant chosen in such a way that $\varphi(0)=1$, we set 
\begin{equation}\label{phi}
\varphi_1(y):=y^\alpha I_\alpha(y)\quad \mbox{and}\quad \varphi_2(y):=y^\alpha Z_\alpha(y).
\end{equation}
\begin{remark}\label{autofun-mu}
We observe that if $\varphi$ satisfies \eqref{eqy}, then the function $\bar\varphi(y):=\varphi(\mu y)$ is a solution of
$$\partial_y(y^{1-2\al} \partial_y \bar\varphi)=y^{1-2\al}\mu^2 \bar \varphi.$$
\end{remark}
\begin{remark}
As said before, $\varphi_2$ is the solution of \eqref{ODE} and in particular it satisfies
\begin{equation}\label{lim-phi0}
\lim_{y\rightarrow 0} \varphi_2(y)=1,
\end{equation}
\begin{equation}\label{lim-phi1}
-\lim_{y\rightarrow 0} y^{1-2\al}\partial_y \varphi_2(y)=	c_\al,
\end{equation}
where $c_\al$ is defined in \eqref{c}.

Moreover, by \eqref{asym-infty} we have
\begin{equation}\label{lim-phi-infty}
\varphi_2(y) \sim y^{\alpha-1/2} e^{-y}\quad \mbox{as}\;\;y\rightarrow \infty.
\end{equation}
%
\end{remark}

\section{The 1-dimensional solution}
In this section we study the following one-dimensional fractional Dirichlet problem:
\begin{equation}\label{AC-1d}
\begin{cases}
(-\partial_{ss})^\al u=F'(u)\quad \mbox{in}\;[0,\la]\\
u(0)=u(\la)=0.
\end{cases}
\end{equation}
In order to do that, we consider the extended problem with $0$-Dirichlet boundary condition on the lateral boundary of the strip $[0,\la]\times (0,+\infty)$. We denote by $(s,y)$ a point in $[0,\la]\times (0,+\infty)$. More precisely, we study existence of nontrivial solutions to the problem
\begin{equation}\label{AC2-1d}
\begin{cases}
\mbox{div}\left(y^{1-2\al} \nabla v\right)=0& \mbox{in}\;[0,\la]\times (0,+\infty)\\
v=0& \mbox{in}\;\partial[0,\la]\times (0,+\infty)\\
- \frac{1}{c_\alpha}\lim_{y\rightarrow 0}y^{1-2\al}\partial_{y}v=F'(v) &\text{on}\; [0,\la]\times \{0\}.
\end{cases}
\end{equation}
The energy functional associated to problem \eqref{AC2-1d} is given by
\begin{equation}\label{en-1d}
E_0(v):=\frac{1}{c_\al}\int_0^{+\infty}\int_0^\la \frac{1}{2}y^{1-2\al}|\nabla v(s,y)|^2ds dy + \int_0^\la F(v(s,0))ds.
\end{equation}
%
%

%
In the following lemma, we give a sufficient and necessary condition on the parameter $\lambda$ for existence of nontrivial solutions to problem \eqref{AC2-1d}.
\begin{lemma}\label{existence-1d}
Let $\lambda_*$ be defined as in \eqref{lambda} and assume that $\lambda>\lambda_*$ is fixed. Then, there exists a nontrivial positive solution of \eqref{AC2-1d} which is a minimizer of $E_0$. Assume that $\lambda \leq \lambda_*$, then there are no positive solutions of \eqref{AC2-1d} and the trivial solution $0$ is the unique minimizer of $E_0$.
\end{lemma}
\begin{proof}
We consider the function
$$w(s,y):=\sin\left(\frac{\pi}{\lambda}s\right)\varphi_2\left(\frac{\pi}{\lambda}y\right),$$ where as before $\varphi_2$ is the solution of \eqref{ODE}. By Remark \ref{autofun-mu} and by  \eqref{lim-phi0} and \eqref{lim-phi1} $w$ satisfies the problem
\begin{equation}\label{eq-w}
\begin{cases}
\mbox{div}\left(y^{1-2\al}\nabla w\right)=0 &\mbox{in} \;[0,\lambda]\times (0,+\infty)\\
w=0 &\mbox{on}\;\partial[0,\lambda]\times (0,+\infty)\\
-\frac{1}{c_\alpha}\lim_{y\rightarrow 0}y^{1-2\al}\partial_y w=\left(\frac{\pi}{\la}\right)^{2\al}w & \mbox{on}\;[0,\lambda]\times \{y=0\}.
\end{cases}
\end{equation}
We use a small multiple of $w$ as a test function to prove that $0$ is not a minimizer when $\lambda>\lambda_*$. First of all, we observe that
$$E_0(0)=\lambda F(0).$$
On the other hand, using Taylor expansion for $F$, we have that
\begin{eqnarray*}
E_0(\eps w)&=&\lambda F(0)+\frac{\eps^2}{2c_\al}\int_0^{+\infty} \int_0^\lambda y^{1-2\al}|\nabla w(s,y)|^2 ds dy \\
&&\hspace{1em}+\frac{\varepsilon^2}{2}\int_0^\la F''(0)w(s,0)^2 ds + \mathcal O(\eps^4).
\end{eqnarray*}
We first observe that
\begin{equation}\label{potential}
\int_0^\lambda F''(0)w(s,0)^2=F''(0)\int_0^\lambda \sin^2\left(\frac{\pi}{\lambda}s \right) ds =\frac{\lambda}{2}F''(0).
\end{equation}
To compute the Dirichlet energy, we use the change of variable $\bar y=\frac{\pi}{\la} y$ and we integrate by parts in $\bar y$ to get
\begin{equation}\label{Dirichlet}
\begin{split}
&\frac{1}{2c_\al}\int_0^{+\infty}\int_0^\la y^{1-2\al}|\nabla w(s,y)|^2ds dy=\\
&\hspace{1em}=\frac{1}{2c_\al} \int_0^{+\infty} \int_0^\la y^{1-2\al}\left( \frac{\pi}{\la}\right)^{2}\left[\cos^2\left(\frac{\pi}{\la}s\right)\varphi_2^2\left(\frac{\pi}{\la}y\right)+
\sin^2\left(\frac{\pi}{\la}s\right)\dot{\varphi_2}^2\left(\frac{\pi}{\la}y\right)\right]dsdy\\
&\hspace{1em}=\frac{1}{c_\al}\frac{\la}{4}\left(\frac{\pi}{\la}\right)^{2\al} \int_0^{+\infty}\bar y^{1-2\al}\left[\varphi_2^2(\bar y) +\dot{\varphi_2}^2(\bar y)\right]d\bar y\\
&\hspace{1em}=\frac{1}{c_\al}\frac{\la}{4}\left(\frac{\pi}{\la}\right)^{2\al} \int_0^{+\infty}\left(\bar y^{1-2\al} \varphi_2^2(\bar y)-\varphi_2(\bar y)\partial_{\bar y}\left(\bar y^{1-2\al}\partial_{\bar y} \varphi_2(\bar y)\right)\right) d\bar y \\
 &\hspace{2em}-\frac{1}{c_\al}\frac{\la}{4}\left(\frac{\pi}{\la}\right)^{2\al} \lim_{\bar y\rightarrow 0}\bar y^{1-2\al} \varphi_2(\bar y)\dot{\varphi_2}(\bar y)\\
&\hspace{1em}=\frac{\la}{4}\left(\frac{\pi}{\la}\right)^{2\al},
\end{split}
\end{equation}
where in the last equality we have used that $\varphi_2$ satisfies \eqref{eqy} and the asymptotic behaviours  \eqref{lim-phi0} and \eqref{lim-phi1}.

Therefore, combining together \eqref{potential} and \eqref{Dirichlet}, we deduce that
\begin{equation*}
\begin{split}
E_0(\eps w)&=\lambda F(0) +\frac{\eps^2}{4}\la \left (\left(\frac{\pi}{\la}\right)^{2\al}+F''(0)\right) + \mathcal O(\eps^4)\\
& <E_0(0),
\end{split}
\end{equation*}
for $\eps$ small enough, provided that $\la>\la_*$. This concludes the proof of the first part of the statement, since we get a nontrivial minimizer for $E_0$, which can be chosen to be positive by standard arguments.

To prove the nonexistence of positive solutions for $\la\leq \la_*$, we multiply the first equation of $\eqref{AC2-1d}$ by $w$ and we integrate by parts, to get
\begin{equation*}
\begin{split}
0&=\frac{1}{2c_\al}\int_0^{+\infty}\int_0^\la \mbox{div}\left(y^{1-2\al}\nabla v(s,y)\right) w(s,y)dsdy \\
&=-\frac{1}{2c_\al}\int_0^{+\infty}\int_0^\la y^{1-2\al} \nabla v(s,y)\cdot\nabla w(s,y) ds dy \\
&\hspace{1em}- \int_0^\la \lim_{y\rightarrow 0} y^{1-2\al}\partial_y v(s,y) w(s,0) ds\\
&= \frac{1}{2c_\al}\int_0^\la \lim_{y\rightarrow 0} y^{1-2\al}\partial_y w(s,y) v(s,y) ds +\int_0^\la F'(v) w(s,0)ds\\
&=\frac{1}{2}\int_0^{\la} w(s,0)\left(F'(v) + \left(\frac{\pi}{\la}\right)^{2\al} v(s,0)\right) ds,
\end{split}
\end{equation*}
where in the last two equalities we have used that $w$ satisfies \eqref{eq-w}.
This concludes the proof of the lemma, indeed by assumpion $F'(v)\geq F''(0) v$ for any $v\geq 0$, and therefore we must have $v(s,0)\equiv 0$ when $\la \leq \la_*$, which implies $v\equiv 0$ by uniqueness of solutions to the extended problem \eqref{Dir-ext}.
\end{proof}
For $\la>\la_*$ we will denote by $v_0$ the nontrivial minimizer of $E_0$, which existence is established in Theorem \ref{existence-1d}. From what we have seen
above, we have the inequality:
\begin{equation}\label{est-1d}
E_0(v_0) < \la F(0).
\end{equation}
\section{The existence result for $\la>\la_*$}

This section is devoted to the proof of Theorem \ref{exist}.

In the following we will use cilindrical coordinates $(r,\theta,s,y)\in [0,+\infty)\times \mathbb S^1\times \R\times \R^+$ to parametrize $\R^3\times \R^+$. In order to find a solution of \eqref{AC} which is invariant under the screw motion $\sigma_\la^\beta$, we will look for a solution $v$ of the extended problem \eqref{AC2} which is invariant under the transformation, that for simplicity of notation we still denote by $\sigma_\la^\beta$, acting on $\R^3\times \R^+$ in the natural way
$$\sigma_\la^\beta(z,s,y)=\left(e^{i\beta}z,s+\frac{\la}{\pi}\beta,y\right).$$
More precisey, we require that
$$v(r,\theta,s,y)=v\left(r,\theta+\beta, s+\frac{\la}{\pi}\beta,y\right)\quad \mbox{for every}\;\;\beta\in \R,$$
which implies that
$$v(r,\theta,s,y)=v(r,\theta, s+2\la,y)\quad \mbox{and}$$
$$v(r,\theta,s,y)=v\left(r,0,s-\frac{\la}{\pi}\theta,y \right).$$
Moreover, we assume that
$$v(r,\theta,-s,y)=-v(r,\theta,s,y).$$
Using these invariances, it is clear that in order to construct the solution $v$, it is enough to construct it for $\theta=0$, that is we just need to know
$$V(r,s,y)=v(r,0,s,y)\quad\mbox{for}\;(r,s,y)\in [0,+\infty)\times [0,\lambda]\times \R^+.$$
Of course, by construction, we have that if $V$ is positive in $(0,+\infty)\times (0,\lambda)\times \R^+$ and vanishes on $\partial \left([0,+\infty)\times [0,\lambda]\right)\times \R^+$, then the zero level set of $v$ is $H_\la \times \R^+$, and therefore the zero level set of its trace $u$ on $\{y=0\}$, that is a solution of \eqref{AC}, is exactly the helicoid $H_\la$.

We need now to write  problem \eqref{AC2} and the energy \eqref{en} in the cilindrical coordinates introduced above for functions invariant under screw motion.
Problem \eqref{AC2} becomes
\begin{equation}\label{eq-cil}
\begin{cases}
V_{rr}+\frac{1}{r}V_r+\left(1+\frac{\la^2}{\pi^2r^2}\right)V_{ss} + V_{yy}+\frac{1-2\al}{y} V_y =0& \text{in}\; [0,+\infty)\times [0,\lambda]\times \R^+,\\
-\frac{1}{c_\al} \lim_{y\rightarrow 0}y^{1-2\al}\partial_{y}V=F'(V)& \text{on}\; \{y=0\}
\end{cases}
\end{equation}

In what follows we will use the following notations:
$$S_{R}:=[0,R]\times [0,\la]\quad \mbox{and}\quad C_{R,L}:=S_R\times [0,L].$$
We define the following subsets of $\partial C_{R,L}$:
$$\partial^+ C_{R,L}:=\overline{\partial C_{R,L}\cap \{y>0\}},$$
$$\partial^0 C_{R,L}:=\partial C_{R,L}\setminus \partial^+ C_{R,L}.$$
The energy associated to problem \eqref{eq-cil} in the cylinder $C_{R,L}$ is given by:
\begin{eqnarray*}
 E(V,C_{R,L})&=&\frac{1}{2c_\alpha}\iint_{C_{R,L}} y^{1-2\al}\left(|V_r|^2+\left(1+\frac{\la^2}{\pi^2r^2}\right)|V_{s}|^2 + |V_y|^2\right) r\: dr\: ds\: dy \\
 &&\hspace{2em}+\int_{\partial^0 C_{R,L}} F(V) r \:dr\: ds.\end{eqnarray*}

Let $H^1(C_{R,L},y^{1-2\al})$ denote the weighted Sobolev space%
$$H^1(C_{R,L},y^{1-2\al})=\{ v: C_{R,L}\rightarrow \R\;|\;y^{1-2\al}(v^2+|\nabla v|^2) \in L^1_r(C_{R,L})\},$$
where $L^1_r$ denotes the space $L^1(C_{R,L})$ with respect to the measure $r \:dr \;ds\; dy$, 
and let $\tilde H_0^1(C_{R,L},y^{1-2\al})$ be the space
$$\tilde H_0^1(C_{R,L},y^{1-2\al})=\{ v\in H^1(C_{R,L},y^{1-2\al})\;|\;|v|\leq 1,\;v \equiv 0 \;\;\mbox{on}\;\;\partial^+C_{R,L}\},$$

We recall that (see the proof of Lemma 4.1 in \cite{CS2}), the inclusion
\begin{equation}\label{compact}
\tilde H^1_0(C_{R,L},y^{1-2\al}) \subset \subset L^2(\partial^0 C_{R,L})
\end{equation}
is compact.

We are now ready to give the prove of Theorem \ref{exist}.

\begin{proof}[Proof of Theorem \ref{exist}]
By  a standard variational argument and using compactness of the inclusion \eqref{compact}, taking a minimizing sequence $\{V^k_{R,L}\}\in \tilde H^1_0(C_{R,L},y^{1-2\al})$ and a subsequence convergent in $L^2(\partial^0 C_{R,L})$, we conclude that $E(\cdot,C_{R,L})$ admits an absolute minimizer $V_{R,L}$ in $\tilde H^1_0(C_{R,L},y^{1-2\al})$. Without loss of generality, by a standard truncation argument, we may assume $0\leq V_{R,L}\leq 1$.

It is easy to check that $V_{R,L}$ is a solution of \eqref{eq-cil} in $C_{R,L}$ (with $0$-Dirichlet boundary condition on $\partial^+ C_{R,L}$). Arguing as in the proof of Theorem 1.3 in \cite{Cinti}, one can prove that $V_{R,L}$ extends to a solution of 
$$
\begin{cases}
\text{div}(y^{1-2\alpha} \nabla v)=0 &\mbox{in}\:\: C_{R,L}\\
v\equiv 0 &\mbox{on}\:\: \partial^+ C_{R,L}\\
-\frac{1}{c_\al}\lim_{y\rightarrow 0}y^{1-2\al}\partial_y v=F'(v)&\mbox{on}\:\: \partial^0 C_{R,L},
\end{cases}
$$
where, for simplicity, we keep the notation $C_{R,L}$ to denote the corresponding cylinder in $\R^4_+$, that is
$$C_{R,L}=\{(x,y)=(r,\theta,s,y)\in \R^4_+\,:\, 0<r<R,\;0<y<L\},$$ and the subsets of its boundary
$$\partial^+ C_{R,L}=\overline{\partial C_{R,L}\cap \R^4_+},\;\; \partial^0 C_{R,L}:=\partial C_{R,L}\setminus \partial^+ C_{R,L}.$$

Some care is needed to show that it is a solution close to $\{r=0\}$, and we refer to \cite{Cinti} for details.
We now wish to pass to the limit in $R$ and $L$, and obtain a
solution in all of $[0,\infty)\times [0,\lambda] \times \R^+$. Let $S>0$, $L'>0$ and consider the family
$\{V_{R,L}\}$ of solutions in $[0,{S+2}]\times [0,\lambda]\times[0,L'+2]$, with
${R>S+2}$ and $L>L'+2$. Since $|V_{R,L}|\leq 1$, regularity results proven in Proposition 4.6 of \cite{CS1} , give a uniform
$C^{2,\alpha}([0,{S}]\times [0,\lambda]\times[0,L'])$ bound for $V_{R,L}$
(uniform with respect to $R$ and $L$). We have
\begin{equation}\label{grad1}
|\nabla V_{R,L}|\leq C \quad\text{ in } [0,{S}]\times [0,\lambda]\times[0,L'],
\qquad\text{for all } R>S+2,\:\:L>L'+2
\end{equation}
for some constant $C$ independent of $S$, $R$, $L$ and $L'$. 
Choose now
$L=R^{b}$, with $1/2<b<1$ (this choice will be used later to prove that the solution that we construct is not identically zero). By the
Ascoli-Arzel\`a Theorem, a subsequence of $\{V_{R,R^{b}}\}$ converges in
$C^2([0,S]\times [0,\lambda]\times [0,S^{b}])$ to a solution in
$([0,S]\times [0,\lambda]\times [0,S^{b}])$. Taking $S=1,2,3,\ldots$ and making a
Cantor diagonal argument, we obtain a sequence $V_{R_j,R_j^{b}}$
converging in $C^2_{{loc}}([0,\infty)\times [0,\lambda]\times \R^+)$ to a solution $V\in
C^2_{{loc}}([0,\infty)\times [0,\lambda]\times \R^+)$. By construction we have found a solution $V$
with $|V|\leq 1$.
%

It only remains to prove that $V$ is positive.

We start by proving that $V$ is not identically zero. In order to do that, following the argument in \cite{Cinti},  we establish an energy estimate for $V$ using a comparison argument, based on the minimality of $V_{R,R^b}$ in the set $C_{R,R^b}$ .

%
%
%

%

Suppose by contradiction that $\la>\la_*$ and that $V\equiv 0$. Then, given $ R>0$, the minimizing sequence $V_{R,R^b}$ converges uniformly to $0$ on $C_{ R, R^b}$ when $R\rightarrow +\infty$. The energy of $0$ in $C_{R,R^b}$ is clearly
$$E(0,C_{ R, R^b})=\frac{\la}{2}F(0)  R^2.$$
Therefore, for any function $W$ vanishing on $\partial^+ C_{ R, R^b}$, we have that
\begin{equation}\label{min-est}
E(W,C_{ R, R^b})\geq E(0,C_{ R, R^b})=\frac{\la}{2}F(0)  R^2.\end{equation}
We build now a suitable competitor $W$ and we arrive to a contradiction with \eqref{min-est}. First we define two smooth cut-off functions $\eta$ and $\xi$ as follows:
$$\eta(r):=\begin{cases} 0 & \mbox{for}\;r\in [0,1/2]\cup [ R-1/2, R]\\
1 &\mbox{for}\;r\in [1, R-1],\end{cases}
$$
and
$$\xi(y)=\begin{cases} 1 & \mbox{if}\:\:0<y\leq  R^{b}- R^{a}\\
\displaystyle \frac{\log { R^b}-\log y}{\log { R^b}-\log {( R^b- R^a)}} & \mbox{if}\:\: R^{b}- R^{a}<y\leq  R^{b}.
                \end{cases}
$$
We set
$$W(r,s,y):=\eta(r)\xi(y)v_0(s,y),$$
where $v_0$ is the minimizer of $E_0$ (see \eqref{en-1d}), whose existence is established in Lemma \ref{existence-1d} and which satisfies \eqref{est-1d}.

We compute now the energy of $W$.
First, we observe that the potential energy is estimated by
\begin{equation}\label{poten}
\int_0^\la \int_0^R F(W(r,s,0)) r dr ds \leq \frac{R^2}{2} \int_0^\la F(v_0(s,0)) ds + \mathcal O (R).
\end{equation}

We estimate now the Dirichlet energy. Since $|\eta|,|\xi|,|v_0|\leq 1$, we have that:
\begin{align}\label{energy}
&\frac{1}{2 c_\alpha}\iint_{C_{R, R^b}} y^{1-2\al}|\nabla W|^2  dx dy \nonumber\\
&\hspace{0.5em} \leq \frac{1}{2 c_\alpha}\iint_{C_{R, R^b}}y^{1-2\al}\left(|\nabla \eta|^2+|\nabla v_0|^2+|\dot{\xi}|^2\right) dx dy. \\
\end{align}
This, together with \eqref{poten}, implies
$$E(W,C_{R,R^b})\leq \frac{R^2}{2} E_0(v_0) + \frac{1}{2 c_\alpha}\iint_{C_{R, R^b}} y^{1-2\al}\left(|\nabla \eta|^2+|\dot{\xi}|^2\right) dx dy + \mathcal O(R).$$
We estimate now the second term on the right-hand side above. In the following, $C_\al$ will denote positive, possibly different, costants depending only on $\al$. Using the definition of $\eta$, we have that
\begin{equation}\label{eta}
\iint_{C_{R, R^b}}y^{1-2\al}|\nabla \eta|^2 dx dy \leq C_\al \la R  \int_0^{ R^{b}} y^{1-2\al} dy=C_\al \la  R^{1+2b(1-\al)}.
\end{equation}
On the other hand, by the definition of $\xi$ and using polar coordinates, we have
\begin{align}\label{xi}
&\int_0^{ R^b}\int_0^\la \int_0^R y^{1-2\al}|\dot \xi (y)|^2 r\: dr ds dy \leq C_\al \la  R^2 \frac{1}{\left(\log{\frac{ R^b}{ R^b- R^a}}\right)^2} \int_{ R^b-  R^a}^{ R^b}\frac{y^{1-2\al}}{y^2}dy\nonumber\\
&\hspace{1em}\leq C_\al \la  R^2 \frac{1}{-\log{\left(1- R^{a-b}\right)^2}} \left[ \frac{1}{ R^{2\al b}- R^{2\al a}}-\frac{1}{ R^{2\al b}}\right]\nonumber\\
&\hspace{1em}\leq C_\al \la  R^2\cdot  R^{2(b-a)}\cdot  R^{-2\al b} =C_\alpha \la  R^{2+2b(1-\al)-2a}.
\end{align}
Therefore, plugging \eqref{poten}, \eqref{eta} and \eqref{xi} into \eqref{energy}, we deduce
\begin{equation}
E(W,C_{ R, R^b})\leq C_\al \la  R^{1+2b(1-\al)}+C_\al \la  R^{2+2b(1-\al)-2a} +\frac{ R^2}{2}E_0(v_0) + \mathcal O(R).
\end{equation}
Now we choose $1/2<a<b<\frac{1}{2(1-\al)}$. With this choice of $b$ and $a$, there exists $\eps=\eps(b)$ such that
$$E(W,C_{ R, R^b})\leq C_\al \la  R^{2-\eps} + \frac{ R^2}{2}E_0(v_0) + \mathcal O(R^{2-\eps}).$$
This, together with \eqref{min-est}, implies
$$\frac{1}{2}F(0)R^2\leq \frac{ R^2}{2}E_0(v_0) + \mathcal O(R).$$
This gives a contradiction, since for $\la>\la_*$ we have that
$$E_0(v_0)<\la F(0),$$
which was established in \eqref{est-1d}.

This implies that $V$ is not identically zero. Since by construction $V\geq 0$, we conclude that $V$ is strictly positive using the Hopf's Lemma.
\end{proof}
\section{The nonexistence result}
%
In this section we give the proof of our nonexistence result Theorem \ref{nonexist}.

In order to do this, we need to establish exponential decay in the $y$-variable of a bounded solution $v$ of \eqref{AC2} which vanishes on $H_\lambda \times \R^+$ and is invariant under screw motion. We stress that in general exponential decay in the $y$ variable for bounded solutions in all the half-space of problem \eqref{AC2} is not true, as one can see in the particular case of the one-dimensional Peierls-Nabarro problem:
$$\begin{cases}
\Delta v=0 &\mbox{in}\;\;\R^2_+\\
\partial_y v(x,0)=\frac{1}{\pi}\sin{(\pi u)}& \mbox{on} \R,
\end{cases}
$$ 
for which an explicit solution is given by $v(x,y)=\frac{2}{\pi}\arctan{\left(\frac{x}{1+y}\right)}$ (see \cite{C-SM} and references therein). On the other hand, as shown in Lemmas \ref{spect-ext} and \ref{autofun1}, the solution of problem \eqref{Dir-ext} in a cylinder $\mathcal C_\Omega=\Omega\times \R^+$, with $\Omega$ bounded, with $0$-Dirichlet condition on the lateral boundary $\partial \Omega \times \R^+$ decays exponentially in $y$.
 In our situation, we are able to prove exponential decay, thanks to the symmetry of the problem and to the invariances of the solution $v$.

\begin{proposition}\label{exp-decay}
Suppose that $v$ is a bounded solution of
$$\begin{cases}
\mbox{div}(y^{1-2\alpha} \nabla v)=0 & \mbox{in}\:\:\R^4_+\\
v(x,0)=u(x),
\end{cases}
$$
which vanishes on $H_\lambda \times \R^+$ and is invariant under screw motion.

Then, there exist a positive constant $K$ such that
\begin{equation}\label{exp}
|v(x,y)|\leq K\varphi_2(y)\sim y^{\alpha-\frac{1}{2}} e^{-y}\quad \mbox{as}\;\;y \rightarrow \infty.
\end{equation}
\end{proposition}
\begin{proof}
Using polar coordinates and the invariance of $v$, as done in Section 4, we have that the function $V(r,s,y)=v(r,0,s,y)$ satisfies:
\begin{equation}\label{eq-cil-exp}
\begin{cases}
V_{rr}+\frac{1}{r}V_r+\left(1+\frac{\la^2}{\pi^2r^2}\right)V_{ss} + V_{yy}+\frac{1-2\al}{y} V_y =0& \text{in}\; [0,+\infty)\times [0,\lambda]\times \R^+,\\
V=0& \text{on}\; s\in\{0,\lambda\},\\
V=u& \text{on}\; \{y=0\}.
\end{cases}
\end{equation}
We define now the following function, which provides an upper barrier for $V$. For $C>e^2$, $K$ large enough to be chosen later, and $\eps >0$, we set
\begin{equation}\label{barrier}
w^\eps(r,s,y):= K\sin\left(\frac{\pi}{\lambda} s\right)\cdot\left[ \varphi_2\left(\frac{\pi}{\lambda}y\right) + \eps \varphi_1\left(\frac{\pi}{\lambda}y\right) + \eps \left(e^{\frac{1}{2}\frac{\pi}{\lambda} r} + C e^{-\frac{1}{2}\frac{\pi}{\lambda} r}\right)\right], 
\end{equation}
where $\varphi_1$ and $\varphi_2$ were defined in \eqref{phi}. 
To conlude the proof, it is enough to prove that, for any $\eps >0$ (sufficiently small),
\begin{equation}\label{upper-bound}
V\leq w^\eps.
\end{equation}
Then, the conclusion follows by sending $\eps \rightarrow 0$.

We start by showing that $w_\eps$ satisfies:
\begin{equation}\label{superharmonic}
w^\eps_{rr}+\frac{1}{r}w^\eps_r+\left(1+\frac{\la^2}{\pi^2r^2}\right)w^\eps_{ss} + w^\eps_{yy}+\frac{1-2\al}{y} w^\eps_y \leq0
\end{equation}
By a direct computation, we have that:
\begin{eqnarray}\label{conti-superharm}
&&\hspace{-1em}\frac{1}{K}\left[w^\eps_{rr}+\frac{1}{r}w^\eps_r+\left(1+\frac{\la^2}{\pi^2r^2}\right)w^\eps_{ss} + w^\eps_{yy}+\frac{1-2\al}{y} w^\eps_y\right]\nonumber\\
&&\hspace{1em}= -\frac{3}{4}\eps \left(\frac{\pi}{\lambda}\right)^2\sin\left(\frac{\pi}{\lambda}s\right) \left(e^{\frac{1}{2}\frac{\pi}{\lambda}r} + C e^{-\frac{1}{2}\frac{\pi}{\lambda}r}\right)\nonumber\\
&&\hspace{2em} + \frac{1}{2r}\eps \frac{\pi}{\lambda}\sin\left(\frac{\pi}{\lambda}s\right) \left(e^{\frac{1}{2}\frac{\pi}{\lambda}r} - C e^{-\frac{1}{2}\frac{\pi}{\lambda}r}\right)- \frac{1}{r^2}w^\eps\nonumber\\
&&\hspace{1em}\leq -\frac{3}{4}\eps \left(\frac{\pi}{\lambda}\right)^2\sin\left(\frac{\pi}{\lambda}s\right) \left(e^{\frac{1}{2}\frac{\pi}{\lambda}r} + C e^{-\frac{1}{2}\frac{\pi}{\lambda}r}\right)\nonumber\\
&&\hspace{2em}+\frac{1}{2r}\eps \frac{\pi}{\lambda}\sin\left(\frac{\pi}{\lambda}s\right) \left(e^{\frac{1}{2}\frac{\pi}{\lambda}r} - C e^{-\frac{1}{2}\frac{\pi}{\lambda}r}\right)=A_1+A_2.
\end{eqnarray}
We study now separately the cases $r\geq \frac{\lambda}{\pi}$ and $r < \frac{\lambda}{\pi}$.

If $r\geq \frac{\lambda}{\pi}$, the last term in \eqref{conti-superharm} is bounded above by 
\begin{eqnarray*}
&&A_2\leq \frac{1}{2}\eps \left(\frac{\pi}{\lambda}\right)^2\sin\left(\frac{\pi}{\lambda}s\right) \left(e^{\frac{1}{2}\frac{\pi}{\lambda}r} - C e^{-\frac{1}{2}\frac{\pi}{\lambda}r}\right)\\
&&\hspace{1em}\leq \frac{1}{2}\eps \left(\frac{\pi}{\lambda}\right)^2\sin\left(\frac{\pi}{\lambda}s\right) \left(e^{\frac{1}{2}\frac{\pi}{\lambda}r} + C e^{-\frac{1}{2}\frac{\pi}{\lambda}r}\right).
\end{eqnarray*}
This, combined with \eqref{conti-superharm} implies \eqref{superharmonic} when $r\geq \frac{\lambda}{\pi}$.

When $r< \frac{\lambda}{\pi}$, we immediatley deduce that 
$$e^{\frac{1}{2}\frac{\pi}{\lambda}r} - C e^{-\frac{1}{2}\frac{\pi}{\lambda}r}\leq 0,$$
since we have chosen $C>e^2$ and therefore \eqref{superharmonic} holds.
It remains ti prove that $w^\eps\geq V$ on $\partial \big([0,+\infty)\times [0,\lambda]\times \R^+ \big)$. On the set $\{s=0\}\cup \{s=\lambda\}$, this is easy, since by definition $w^\eps=0=V$ while on $\{r=0\}$ we have $w^\eps \geq 0=V$. When $y\rightarrow \infty$ and $r\rightarrow \infty$, it is also true. Indeed $V$ is bounded and, for any fixed $\eps$, $w^\eps$ can be made arbitrarly large for $y$ and $r$ sufficiently large. To conclude we just have to prove that $w^\eps \geq V$ on $\{y=0\}$. On this part of the boundary, we have for $\eps$ sufficienlty small
$$w^\eps(r,s,0)\geq \frac{K}{2}\sin\left(\frac{\pi}{\lambda}s\right).$$
By Proposition 4.6 in \cite{CS1}, we know that $v$ has bounded gradient, in particular $v_s=V_s$ si bounded.
Therefore, since $V=0$ on $\{s=0\}\cup \{s=\lambda\}$, we deduce that, there exists a constant $\tilde C$ such that
$$|V(r,s,0)|\leq \tilde C \min\{s, \lambda -s\}.$$
To conclude, we observe that it is possible to choose $K$ sufficiently large, so that
$$w^\eps(r,s,0)\geq \frac{K}{2}\sin\left(\frac{\pi}{\lambda}s\right)\geq \tilde C \min\{s, \lambda -s\}\geq |V(r,s,0)|.$$

We have proven that $w^\eps$ is an upper barrier for $V$ and therefore \eqref{upper-bound} holds. This concludes the proof of the proposition.
\end{proof}
We can now prove our non-existence result.
\begin{proof}[Proof of Theorem \ref{nonexist}]
We start by observing that, by uniqueness of solutions to problem \eqref{Dir-ext} and since the operator $\mbox{div}(y^{1-2\al} \nabla)$ is invariant under the screw motion $\sigma_\la^\beta$, in order to prove Theorem \ref{nonexist}, it is enough to prove the corresponding non-existence result for the extended problem \eqref{AC2}.
We write $V(r,s,y)=v(r,0,s,y)$ and we consider problem \eqref{eq-cil} written in cylindrical coordinates. Let $\eta$ be a cut-off function only depending on $r$ such that
$$\eta(r)=\begin{cases} 1 & r\leq R\\
0 & r\geq 2R
\end{cases} \quad \mbox{and}\quad |\nabla \eta|\leq \frac{C}{R}.
$$
 We multiply \eqref{eq-cil} by $V\eta^2$ and we integrate by parts,  to obtain
 \begin{eqnarray}\label{nonex-ineq1}
 &&\int_0^{+\infty}\int y^{1-2\alpha}\left( |V_r|^2+\left(1 +\frac{\lambda^2}{\pi^2 r^2}\right) |V_s|^2+ |V_y|^2\right)\eta^2\:r\:dr\:ds\:dy\nonumber \\
 &&\hspace{1em} + \int F'(V)V\eta^2 \:r\:dr\:ds = 2\int_0^{+\infty}\int y^{1-2\alpha} V\: V_r\: \eta\:\eta_r\:r \:dr\:ds\:dy,
 \end{eqnarray}
where the domain of integration is $[0,\infty)\times [0,\lambda]$ where it is not explicitely written.
Now we use the assumption $F'(t)\:t\geq F''(0)\:t^2$ for any $t\in \R$, to get
\begin{eqnarray*}&&\int_0^{\infty} \int_0^\lambda y^{1-2\alpha}\big(|V_s|^2+ |V_y|^2\big)ds dy + \int_0^\lambda F'(V)V ds\\
 &&\hspace{1em} \geq \int_0^{\infty} \int_0^\lambda y^{1-2\alpha}\big(|V_s|^2+ |V_y|^2\big)ds dy + \int_0^\lambda F''(0)V^2 ds.\end{eqnarray*}
Using a truncation argument in the $y$-variable (in order to make $V$ compactly supported in $y$), the exponential decay established in Proposition \eqref{exp-decay} and the fact that $0$ is the unique absolute minimizer for the one-dimensional problem when  $\lambda \leq \lambda_*$, we get
$$
\int_0^{\infty} \int_0^\lambda y^{1-2\alpha}\big(|V_s|^2+ |V_y|^2\big)ds dy + \int_0^\lambda F''(0)V^2 ds \geq 0.
$$
Therefore, from \eqref{nonex-ineq1} we deduce
$$\int_0^\infty \int y^{1-2\alpha} |V_r|^2\eta^2\:r\:dr\:ds\:dy \leq 2\int_0^\infty \int y^{1-2\alpha} V\:V_r\:\eta\:\eta_r\:r\:dr\:ds\:dy.$$
Using Cauchy-Schwarz inequality on the right-hand side we have
\begin{eqnarray}\label{nonex-ineq2}
&&\int_0^\infty \int y^{1-2\alpha} |V_r|^2\eta^2\:r\:dr\:ds\:dy\leq \left(\int_0^\infty \int_{r\in[R,2R]} y^{1-2\alpha} |V_r|^2\eta^2\:r\:dr\:ds\:dy\right)^{\frac{1}{2}}\cdot\nonumber \\
&&\hspace{2em}\cdot \left(\int_0^\infty \int_{r\in[R,2R]} y^{1-2\alpha} |V|^2|\eta_r|^2\:r\:dr\:ds\:dy\right)^{\frac{1}{2}}
\end{eqnarray}
Now, using that $|\nabla \eta|\leq \frac{C}{R}$ and  the exponential decay \eqref{exp} of $V$ in the variable $y$, we deduce that the second integral on the right-hand side is bounded, which implies that the integral
$$\int_0^\infty \int_{r\in[R,2R]} y^{1-2\alpha} |V_r|^2\eta^2\:r\:dr\:ds\:dy$$
is bounded independently of $R$. Letting $R$ tend to infinity, we conclude that
$$\int_0^\infty \int y^{1-2\alpha} |V_r|^2\:r\:dr\:ds\:dy\leq C.$$
In particular, there exists a sequence $R_i\rightarrow \infty$ for which
$$\lim_{i\rightarrow \infty}\int_0^\infty \int_{r\in[R_i,2R_i]} y^{1-2\alpha} |V_r|^2\eta^2\:r\:dr\:ds\:dy=0$$
This, together with \eqref{nonex-ineq2}, implies that
$$\int_0^\infty \int y^{1-2\alpha}|V_r|^2\:r\:dr\:ds\:dy \leq 0,$$
which concludes the proof.
\end{proof}

\section{Proof of Theorem \ref{helicoid}}
We can give the proof of our last result.

\begin{proof}[Proof of Theorem \ref{helicoid}]
We have to show that
$$\mathcal H^{2\al}_{H_\la} (x_0) =0$$
for all $x_0 \in H_\lambda$. We can assume for this calculation that $x_0=(t_0,0,0) = (t_0 e^{i \cdot 0},0)$ where $t_0>0$.

We can write
$$
\R^3 \setminus H_\lambda = E_+ \cup E_-
$$
where
\begin{align*}
E_+ & =
\{ \ (te^{i\theta} ,\frac{\lambda}{\pi}(\theta+z) ) \ | \ t>0, \ \theta\in \R, \ z \in (0,\pi) \ \}
\\
E_- & =
\{ \ (te^{i\theta} , \frac{\lambda}{\pi}(\theta+z) ) \ | \ t<0, \ \theta\in \R, \ z \in (0,\pi) \ \}
\end{align*}
are the two connected components of $\R^3 \setminus H_\lambda$.
Then
\begin{align*}
\mathcal H^{2\al}_{H_\la} (x_0)
= \int_{\R^3} \frac{\chi_{E_+}(x)  -\chi_{E_-}(x)  }{|x-x_0|^{3+2\alpha}}\, dx .
\end{align*}

Consider the transformation
$$
f(t e^{i\theta} , \frac{\lambda}{\pi}(\theta+z) ) =
(t e^{-i\theta} , -\frac{\lambda}{\pi}(\theta+z) )
$$
defined for $(t e^{i\theta} , \frac{\lambda}{\pi}(\theta+z) )\in E_+$.
Let us verify  that $(t e^{-i\theta} , -\frac{\lambda}{\pi}(\theta+z) )\in E_-$ if $(t e^{i\theta} , \frac{\lambda}{\pi}(\theta+z) ) \in E_+$. Indeed
$$
( t e^{-i\theta} ,  -\frac{\lambda}{\pi}(\theta+z) )= ( -t e^{i(-\pi-\theta)} , \frac\lambda\pi(-\pi-\theta  + (\pi-z) )) \in E_-
$$
if $t>0$ and $z \in (0,\pi)$. Moreover $f$ is a bijection from $E_+$ onto $E_-$, which preserves volume. Moreover, writing $x = (t e^{i\theta} , \frac{\lambda}{\pi}(\theta+z) )\in E_+$, $x^* = f(x) =(t e^{-i\theta} , -\frac{\lambda}{\pi}(\theta+z) ) \in E_-$ and $x_0 = (t_0,0,0)$, we have
\begin{align*}
|x-x_0|^2
& = (t\cos(\theta)-t_0)^2 + t^2 \sin^2(\theta) + \frac{\lambda^2}{\pi^2}( \theta+z)^2
\\
&=|x^*-x_0|^2 .
\end{align*}
Therefore, changing $x$ to $x^*=f(x)$ in the second integral below
\begin{align*}
\mathcal H^{2\al}_{H_\la} (x_0)
&= \lim_{r\to0}\int_{\{|x-x_0|\geq r\}} \frac{\chi_{E_+}(x)  -\chi_{E_-}(x)  }{|x-x_0|^{3+2\alpha}}\, dx
\\
&= \lim_{r\to0}\left(
\int_{\{|x-x_0|\geq r\}\cap E_+}  \frac{1}{|x-x_0|^{3+2\alpha}}\, dx
-
\int_{\{|x-x_0|\geq r\}\cap E_-}  \frac{1}{|x-x_0|^{3+2\alpha}}\, dx
\right)
\\
&= \lim_{r\to0}\left(
\int_{\{|x-x_0|\geq r\}\cap E_+}  \frac{1}{|x-x_0|^{3+2\alpha}}\, dx
-
\int_{\{|x-x_0|\geq r\}\cap E_+}  \frac{1}{|x^*-x_0|^{3+2\alpha}}\, dx^*
\right)
\\
&=0.
\end{align*}
\end{proof}


\begin{thebibliography}{90}


\bibitem{ABS}
G. Alberti, G. Bouchitt{\'e}, and S. Seppecher,
 {\em Phase transition with the line-tension effect},
Arch. Rational Mech. Anal.
 {\bf 144} (1998), 1--46.

\bibitem{BCdP}
 C. Br\"andle, E. Colorado, A. de Pablo, and U. Sanchez, {\em A concave-convex elliptic problem involving the fractional Laplacian}. Proc. Roy. Soc. Edinburgh Sect. A {\bf 143} (2013), no. 1, 39–71.

\bibitem{BGS}
V. Banica, M.d.M. Gonzalez. and M. Saez,
{\em Some constructions for the fractional Laplacian on noncompact manifolds},  to appear in Rev. Mat. Iber.

\bibitem{CS1} X. Cabr\'e and Y. Sire, {\em  Nonlinear equations for fractional Laplacians, I: Regularity, maximum principles, and Hamiltonian estimates}, Ann. Inst. H. Poincaré Anal. Non Linéaire {\bf 31} (2014), no. 1, 23--53.

    \bibitem{CS2}   X. Cabr\'e and Y. Sire,  {\em Nonlinear equations for fractional Laplacians II: Existence, uniqueness, and qualitative properties of solutions}, Trans. Amer. Math. Soc. {\bf 367} (2015), no. 2, 911--941

\bibitem{CFSMW} X. Cabr\'e, M. Fall, J. Sol\`a-Morales, T. Weth, {\em Curves and surfaces with constant nonlocal mean 
curvature: meeting Alexandrov and Delaunay}, preprint, (Available at http://arxiv.org/pdf/1503.00469.pdf).

\bibitem{C-SM}
X. Cabr\'e and J. Sol\`a-Morales, {\em Layer solutions in a
halph-space for boundary reactions}, Comm. Pure and Appl. Math.
{\bf 58} (2005), 1678--1732.

\bibitem{CT} X. Cabr\'e and J. Tan, {\em Positive solutions of nonlinear problems involving the square root of the Laplacian},
Adv. Math. {\bf 224} (2010), no. 5, 2052--2093.

\bibitem{CRS} L. Caffarelli, J-M. Roquejoffre, and O. Savin, {\em Nonlocal minimal surfaces},
Comm. Pure Appl. Math. {\bf 63} (2010), 1111--1144.

\bibitem{CS}
L. Caffarelli and L. Silvestre, {\em An extension problem related to the
fractional Laplacian}, Comm. Part. Diff. Eq. {\bf 32} (2007),
1245--1260.


\bibitem{C-Val}
L. Caffarelli and E. Valdinoci, {\em Regularity properties of nonlocal minimal surfaces via limiting arguments},
preprint, arXiv: 1105.1158.

\bibitem{Cinti}
E. Cinti, {\em Saddle-shaped solutions of bistable elliptic equations
involving the half-Laplacian}, Ann. Sc. Norm. Super. Pisa Cl. Sci. (5) {\bf 12}, (2013), no. 3, 623--664.

\bibitem{CFMN} G. Ciraolo, A. Figalli, F. Maggi, M. Novaga, {\em Rigidity and sharp stabilit y estimates for hypersurfaces with constant and almost-constant nonlocal mean curvature}, preprint (Available at: http://arxiv.org/pdf/1503.00653.pdf).

\bibitem{DDDV} J. Davila, M. del Pino, S. Dipierro, E. Valdinoci, {\em Nonlocal Delaunay surfaces}, preprint. (Available at: http://arxiv.org/pdf/1501.07459.pdf)
\bibitem{DDW}
J. Davila, M. del Pino, and J. Wei, {\em Nonlocal $s$-minimal surfaces and Lawson Cones}, preprint. (Available at: http://arxiv.org/pdf/1402.4173.pdf).

\bibitem{DPMP}
M. del Pino, M. Musso and F. Pacard, {\em Solutions of the Allen-Cahn equation which are invariant under screw-motion},
Manuscripta Math. {\bf 138}, (2012), no. 3--4, 273--286.


\bibitem{MdM}
M.d.M. Gonz\'alez, {\em Gamma convergence of an energy functional related to the fractional Laplacian},
Calc. Var. Part. Diff. Eq. {\bf 36} (2009), 173--210.

\bibitem{MM} L. Modica and S.Mortola, {\em Un esempio di Γ−-convergenza} (Italian), Boll. Un. Mat. Ital. B (5) 14 (1977), no. 1, 285--299.

\bibitem{SV}
O. Savin and E. Valdinoci, {\em $\Gamma$-convergence for nonlocal phase transitions},
 Ann. Inst. H. Poincaré Anal. Non Lineaire {\bf 29}, (2012), no. 4, 479--500.


\end{thebibliography}
\end{document}